\documentclass{article}
\usepackage{eurosym}
\usepackage{amsmath,amsthm,amssymb}
\usepackage[hmargin=3cm,vmargin=3cm]{geometry}
\usepackage{graphicx}

\usepackage{color}
\theoremstyle{theorem}
\newtheorem{theorem}{Theorem}[section]
\newtheorem{lemma}{Lemma}[section]
\newtheorem{proposition}{Proposition}[section]
\newtheorem{corollary}{Corollary}[section]
\newtheorem{remark}{Remark}[section]
\theoremstyle{definition}
\newtheorem{definition}{Definition}[section]
\newtheorem{example}{Example}[section]
\theoremstyle{remark}
\newcommand{\be}{\begin{equation}}
\newcommand{\ee}{\end{equation}}
\newcommand{\bthe}{\begin{theorem}}
\newcommand{\ethe}{\end{theorem}}
\newcommand{\blem}{\begin{lemma}}
\newcommand{\elem}{\end{lemma}}
\newcommand{\bcor}{\begin{corollary}}
\newcommand{\ecor}{\end{corollary}}
\newcommand{\bdefi}{\begin{definition}}
\newcommand{\edefi}{\end{definition}}
\newcommand{\bp}{\begin{proof}}
\newcommand{\ep}{\end{proof}}

\begin{document}

\title{\textbf{Essential Corrigenda to "Generalized approximation spaces generation from $\mathbb{I}_{j}$-neighborhoods and
ideals with application to Chikungunya disease"}}
\author{Rodyna A. Hosny$^{(1)}$, Naglaa M. Madbouly$^{(2)}$, Mostafa K. El-Bably$^{(3, 4)}$ \\
$^{(1)}${\footnotesize {Department of Mathematics, Faculty of Science,
Zagazig University, Zagazig 44519, Egypt}}\\
$^{(2)}${\footnotesize{Mathematics Department, Faculty of Science, Helwan University, Helwan 11795, Egypt}}\\
$^{(3)}${\footnotesize {Department of Mathematics, Faculty of Science, Tanta
University, Tanta 31527, Egypt}}\\
$^{(4)}${\footnotesize{Jadara University Research Center, Jadara University, Irbid 21110 , Jordan}}}
\date{}
\maketitle

\begin{abstract}
In the work "Generalized approximation spaces generation from $\mathbb{I}_{j}$-neighborhoods and ideals with application to Chikungunya disease, published in \emph{AIMS Mathematics},  \textbf{9}(4) (2024), 10050$-$10077," Al-Shami and Hosny introduced a novel approach for generating generalized neighborhoods through $\mathbb{I}_{j}$-neighborhoods with ideals, subsequently deriving new methods for generalized approximations based on these neighborhoods. However, several errors have been identified in their results, concepts, and methods, along with significant inaccuracies in the provided examples and comparison tables. This paper aims to address and correct these errors, providing counterexamples to illustrate the inaccuracies in the proposed results.  Furthermore, we correct several flawed proofs and present the revised formulations of these results and concepts. Additionally, we offer properties and clarifications to enhance the understanding of these concepts.
\color{black}
\end{abstract}



\section{Introduction}

\begin{itemize}
  \item \textbf{Literature Review:}
\end{itemize}

The theory of rough sets is a significant and effective methodology for handling uncertainty and ambiguity in data, as well as for analyzing and extracting knowledge from it. Introduced by Pawlak in 1982 \cite{Pawlak 1982}, this theory has served as a foundational framework for effective solutions to decision-making problems. Despite its strengths, the constraints of equivalence relation conditions central to this approach have led many researchers to propose generalizations and extensions to broaden its applications.

Several avenues have been explored in this regard. For instance, Yao \cite{Yao 1998} extended the concept of an equivalence relation to a more general binary relation, introducing  the notions of right and left neighborhoods derived from the concepts of after- and before- sets, which were defined in \cite{DeBaets1993}. Similarly, Allam et al. \cite{Allam2005, Allam2006} proposed minimal right and minimal left neighborhoods to generate new approximations that generalize Pawlak's rough set theory.

In 2007, Abo Khadra et al. \cite{khadra2007} introduced a technique to generalize Pawlak's approximation space through the use of binary relations. This method constructs dual topologies directly from right and left neighborhoods, without modifying the data within the information system. It asserts that the class
\[
\mathcal{T} = \{\mathcal{A} \subseteq U : \mho(s) \subseteq \mathcal{A}, \ \forall s \in \mathcal{A} \}
\]
forms a topology on \( U \), where \( U \) represents the universe and \( \mho(s) \) is the neighborhood of \( s \). By defining these neighborhoods, various topologies on \( U \) can be generated.

This technique was expanded in El-Bably's 2008 Master’s thesis \cite{El-BablyTh2008}, which examined additional topological structures and rough set theory generalizations. Later, Abd El-Monsef et al. \cite{Abd El-Monsef 2014} extended the work of \cite{khadra2007, El-BablyTh2008} by introducing a new space called $j$-neighborhood space ($j$-$\textbf{NS}$), constructed by defining union and intersection neighborhoods from right and minimal right neighborhoods. This development led to eight types of neighborhoods derived from binary relations, which have been widely used in recent studies to define new neighborhoods based on these concepts (e.g., \cite{Kaur 2024,Al-shami 2021b}), as well as maximal neighborhoods \cite{Al-shami 2021,R. A. Hosny 2021a,Hosny 2024} and initial neighborhoods \cite{RAHosny 2024}.

Yao \cite{Yao 1998} presented a special class of neighborhood systems, termed 1-neighborhood systems, which combine right and left neighborhoods without studying topological structures. Kozae et al. \cite{Kozae 2010} further proposed $<x>R<x>$ as the intersection of minimal right and left neighborhoods to define rough sets, building on the works in \cite{Allam2005, Allam2006}.

In a subsequent study, Atef et al. \cite{Atef 2020} introduced the $j$-adhesion neighborhood space, highlighting the potential of $j$-$\textbf{NS}$ to extend concepts related to adhesion sets. Later, El-Bably et al. \cite{El-Bably 2021} addressed inconsistencies in the findings of Atef et al., providing corrected results and additional insights into $j$-adhesion neighborhoods. Their approach is based on the concept of "adhesion sets" proposed by Ma in 2012 \cite{Ma 2012}, which is relevant to the theory of covering rough sets.
Additionally, Nawar et al. \cite{Nawar 2021} introduced $j$-adhesion neighborhoods in covering-based rough sets, based on the generalized covering approximation space of Abd El-Monsef et al. \cite{Abd El-Monsef 2015}. Core neighborhood ideas from \cite{Hung 2008,Xiaole 2013} were also extended in \cite{Mareay 2016} using techniques from \cite{khadra2007,Abd El-Monsef 2014}. Furthermore, El-Bably and Al-Shami \cite{El-Bably 2021} introduced core minimal neighborhoods through binary relations, extending Pawlak rough sets to generalized types for applications, such as in lung cancer research. \\

On the other hand, the application of neighborhood has expanded in new directions, such as advanced neighborhood-based methods used for feature selection. For example, studies such as \cite{Zhang 2024, Zhang 2022, Liu 2022, Xiao 2022} have employed these methods for feature selection.\\

\begin{itemize}
  \item \textbf{Motivations and Objectives:}
\end{itemize}

Recently, Al-shami and Hosny \cite{Al-shami 2024} defined a new type of neighborhood based on the $\mathbb{I}_{j}$-neighborhood concept \cite{Al-shami 2021}. Utilizing Abd El-Monsef et al.'s \cite{Abd El-Monsef 2014} definitions of intersection and union neighborhoods, they generated eight novel neighborhoods, studied their properties and relationships, and provided theoretical results supported by illustrative examples. They also developed various rough set approximations based on $\mathbb{I}^{\mathcal{K}}_{j}$-neighborhoods, demonstrating their effectiveness through a medical case study.

Our paper seeks to address several errors, inconsistencies, and inaccuracies identified in the work of Al-shami and Hosny \cite{Al-shami 2024}. The primary objectives of our study are as follows:
\begin{enumerate}
  \item Identify and highlight the errors in Al-shami and Hosny’s results such as Theorem 3.4, Proposition 4.3, and Proposition 5.4 from \cite{Al-shami 2024}.
  \item Provide counterexamples to demonstrate the inaccuracies in their findings, for instance Examples \ref{e1}, \ref{e18}, \ref{e4}, \ref{e5}, \ref{e2} and \ref{e3}.
  \item Correct erroneous results, definitions, and properties introduced in their research.
  \item Rectify incorrect examples and proofs provided in their paper.
  \item Enhance research understanding by providing corrected results, expanding its potential applications.
  \item Contribute to refining rough set theory and methodologies using corrected neighborhood concepts.
\end{enumerate}

\begin{itemize}
  \item \textbf{Organization of the Manuscript:}
\end{itemize}

The remainder of this paper is organized as follows: Section 2 presents essential concepts and foundational  results needed for understanding the manuscript, including corrections to inaccurate definitions, results, and citations from \cite{Al-shami 2024} for coherence. Section 3 identifies and analyzes the errors found in \cite{Al-shami 2024}, specifically addressing inaccuracies in results, definitions, and examples, along with the corresponding corrections. Finally, Section 4 concludes with a summary and discussion of our findings.

\section{Preliminaries}

This foundational section presents essential concepts and core results necessary for understanding both the study and the manuscript’s content. The primary aim here is to correct inaccuracies in definitions and results, as well as to address erroneous citations of key concepts from \cite{Al-shami 2024}, ensuring continuity and coherence throughout the study.\\

First, it is well established that a binary relation $R$ on a nonempty set (universe) $U$ is a subset of $U \times U$. Henceforth, we assume $U$ to be a nonempty finite set, with $R$ representing an arbitrary relation unless stated otherwise. For $s, t \in U$, we write $sRt$ if $(s, t) \in R$. Notably, the following concepts were inaccurately presented in \cite{Al-shami 2024}; we therefore provide the corrected definitions below:
\color{black}

\begin{definition} \label{d1} \cite{Yao 1998,DeBaets1993} A relation $R$ on $U$ is said to be:
\begin{enumerate}
\item serial, if for each $s \in U$ there exists $t \in U$ such that $sRt$.
\item reflexive, if $sRs$, for all $s \in U$.
\item symmetric, if $sRt$ $\Leftrightarrow$ $tRs$.
\item transitive, if $sRt$ whenever $pRt$ and $sRp$.
\item preorder (or quasi-order), if it is reflexive and transitive.
\item equivalence, if it is reflexive, symmetric, and transitive.
\end{enumerate}
\end{definition}


In \cite{Al-shami 2024}, the authors presented Definition 2.2 but omitted the concepts of intersection  (and minimal intersection) neighborhoods and  union  (and minimal union) neighborhoods, which are essential for their definitions and results. To address this, we provide the necessary definitions below.

\begin{definition} \label{d2} The following $\omega$-neighborhoods of an element $s \in U$, inspired by the relation $R$, are
defined as follows:
\begin{enumerate}
\item  after neighborhood (or right neighborhood) of a point $s$, as defined in \cite{Yao 1998,DeBaets1993}, is denoted by $\omega_{a}(s)$ and is given by $\omega_{a}(s) = \{t \in U : sRt\}$.
\item before neighborhood (or left neighborhood) of a point $s$, as defined in \cite{Yao 1998,DeBaets1993}, is denoted by $\omega_{b}(s)$ and is given by $\omega_{b}(s) = \{t \in U : tRs\}$.
\item  minimal-after neighborhood of $s$, as established in \cite{Allam2005}, is denoted by $\omega_{<a>}(s)$ and  is  defined as the intersection of all after neighborhoods contain $s$.
\item  minimal-before neighborhood of $s$, as established in \cite{Allam2005, Allam2006}, is denoted by $\omega_{<b>}(s)$  and  is  defined as  the intersection of all before neighborhoods contain $s$.
\item  intersection neighborhood of a point $s$, described in \cite{Yao 1998,Abd El-Monsef 2014}, is denoted by $\omega_{i}(s)$ and is defined as $\omega_{i}(s) =$ $\omega_{a}(s)$ $\bigcap$ $\omega_{b}(s)$.
\item union neighborhood of a point $s$, as defined in \cite{Yao 1998,Abd El-Monsef 2014}, is denoted by $\omega_{u}(s)$ and is expressed as $\omega_{u}(s) = $ $\omega_{a}(s)$ $\bigcup$ $\omega_{b}(s)$.
\item  minimal-intersection neighborhood of a point $s$, as defined in \cite{Abd El-Monsef 2014,Kozae 2010}, is denoted by $\omega_{<i>}(s)$ and is expressed as $\omega_{<i>}(s) =$ $\omega_{<a>}(s)$ $\bigcap$ $\omega_{<b>}(s)$.
\item  minimal-union neighborhood of a point $s$, referenced in \cite{Abd El-Monsef 2014}, is denoted by $\omega_{<u>}(s)$ and is expressed as  $\omega_{<u>}(s) = $ $\omega_{<a>}(s)$ $\bigcup$ $\omega_{<b>}(s)$. \newline

For simplicity, the set $\{a, b, <a>, <b>, i, u, <i>, <u>\}$ will be denoted by $\Omega$.

\end{enumerate}
\end{definition}

\begin{definition} \label{d5}  \cite{Yao 1998, Allam2005,Kozae 2010} For the $\omega$-neighborhoods and for each $j \in \Omega$, the approximation operators (both lower and upper), boundary region, and measures of accuracy for a nonempty subset $F$ of $U$ are respectively given by:

\begin{center}
$R_{\star}^{\omega_j}(F) = \{s \in U : \omega_{j}(s) \subseteq F\}$.\\
$R^{\star\omega_j}(F) = \{s \in U : \omega_{j}(s) \cap F \neq\emptyset\}$.\\
$BND_{R}^{\star\omega_j}(F) = R^{\star\omega_j}(F) \setminus R_{\star}^{\omega_j}(F)$.\\
$ACC_{R}^{\star\omega_j}(F) = \frac{\mid R_{\star}^{\omega_j}(F)\cap F\mid}{\mid R^{\star\omega_j}(F)\cup F\mid}$, where $\mid R^{\star\omega_j}(F)\cup F\mid\neq0$.
\end{center}
\end{definition}
\color{black}

\begin{definition} \label{d4}  \cite{Atef 2020,El-Bably 2021, Nawar 2021,Hung 2008,Xiaole 2013,Mareay 2016} The following $\rho$-neighborhoods of an element $s \in U$, inspired from a relation $R$, are defined as follows:
\begin{enumerate}
\item  $\rho_{a}(s) = \{t \in U : \omega_{a}(t) = \omega_{a}(s)\}$.
\item  $\rho_{b}(s) = \{t \in U : \omega_{b}(t) = \omega_{b}(s)\}$.
\item  $\rho_{i}(s) = \rho_{a}(s) \cap \rho_{b}(s)$.
\item  $\rho_{u}(s) = \rho_{a}(s) \cup \rho_{b}(s)$.
\item  $\rho_{<a>}(s) = \{t \in U : \omega_{<a>}(t) = \omega_{<a>}(s)\}$.
\item  $\rho_{<b>}(s) = \{t \in U : \omega_{<b>}(t) = \omega_{<b>}(s)\}$.
\item  $\rho_{<i>}(s) = \rho_{<a>}(s) \cap \rho_{<b>}(s)$.
\item  $\rho_{<u>}(s) = \rho_{<a>}(s) \cup \rho_{<b>}(s)$.
\end{enumerate}
\end{definition}

\begin{definition} \label{d9} \cite{Atef 2020, El-Bably 2021, Nawar 2021} For $\rho$-neighborhoods and for each $j \in \Omega$, the approximation operators (both lower and upper), boundary region, and measures of accuracy for a nonempty subset $F$ of $U$ are respectively defined as follows:

\begin{center}
$R_{\star}^{\rho_j}(F) = \{s \in U : \rho_{j}(s) \subseteq F\}$.\\
$R^{\star\rho_j}(F) = \{s \in U : \rho_{j}(s) \cap F \neq\emptyset\}$.\\
$BND_{R}^{\star\rho_j}(F) = R^{\star\rho_j}(F) \setminus R_{\star}^{\rho_j}(F)$.\\
$ACC_{R}^{\star\rho_j}(F) = \frac{\mid R_{\star}^{\rho_j}(F)\mid}{\mid R^{\star\rho_j}(F)\mid}$, where $\mid R^{\star\rho_j}(F)\mid\neq0$.
\color{black}
\end{center}
\end{definition}

\begin{definition} \label{d3} \cite{Al-shami 2021} The subsequent $\mathbb{I}$-neighborhoods of an element $s \in U$, deduced from a relation $R$, are specified as follows:
\begin{enumerate}
\item $\mathbb{I}_{a}(s) = \{t \in U : \omega_{a}(t) \cap \omega_{a}(s) \neq\emptyset\}$.
\item  $\mathbb{I}_{b}(s) = \{t \in U : \omega_{b}(t) \cap \omega_{b}(s) \neq\emptyset\}$.
\item  $\mathbb{I}_{i}(s) = \mathbb{I}_{a}(s) \cap \mathbb{I}_{b}(s)$.
\item  $\mathbb{I}_{u}(s) = \mathbb{I}_{a}(s) \cup \mathbb{I}_{b}(s)$.
\item  $\mathbb{I}_{<a>}(s) = \{t \in U : \omega_{<a>}(t) \cap \omega_{<a>}(s) \neq\emptyset\}$.
\item  $\mathbb{I}_{<b>}(s) = \{t \in U : \omega_{<b>}(t) \cap \omega_{<b>}(s) \neq\emptyset\}$.
\item  $\mathbb{I}_{<i>}(s) = \mathbb{I}_{<a>}(s) \cap \mathbb{I}_{<b>}(s)$.
\item  $\mathbb{I}_{<u>}(s) = \mathbb{I}_{<a>}(s) \cup \mathbb{I}_{<b>}(s)$.
\end{enumerate}
\end{definition}

In \cite{Al-shami 2021}, $\mathbb{I}_{j}$-neighborhoods were studied under the name “$E_{j}$-neighborhoods”, $j \in \Omega$.

\begin{definition} \label{d6} \cite{Al-shami 2021} For $\mathbb{I}$-neighborhoods and for each $j \in \Omega$, the approximation operators (both lower and upper), boundary region, and measures of accuracy of a nonempty subset $F$ of $U$ are defined as follows:

\begin{center}
$R_{\star}^{\mathbb{I}_j}(F) = \{s \in U : \mathbb{I}_{j}(s) \subseteq F\}$.\\
$R^{\star\mathbb{I}_j}(F) = \{s \in U : \mathbb{I}_{j}(s) \cap F \neq\emptyset\}$.\\
$BND_{R}^{\star\mathbb{I}_j}(F) = R^{\star\mathbb{I}_j}(F) \setminus R_{\star}^{\mathbb{I}_j}(F)$.\\
$ACC_{R}^{\star\mathbb{I}_j}(F) = \frac{\mid R_{\star}^{\mathbb{I}_j}(F)\cap F\mid}{\mid R^{\star\mathbb{I}_j}(F)\cup F\mid}$, where $\mid R^{\star\mathbb{I}_j}(F)\cup F\mid\neq0$.
\color{black}
\end{center}
\end{definition}

\begin{definition} \label{d7} \cite{Kuratowski 1966} An ideal $\mathcal{K}$ on a nonempty set $U$ is a nonempty collection of subsets of $U$ that is closed under finite unions and contains all subsets of its elements.
\end{definition}

\begin{definition} \label{d8}  \cite{R. A. Hosny 2021a} Let $R$ and $\mathcal{K}$ represent a binary relation and an ideal, respectively, on a nonempty set $U$. The approximation operators (lower and upper), boundary region, and accuracy of a nonempty subset $F$ of $U$, derived from $R$ and $\mathcal{K}$ using $\mathbb{I}$-neighborhoods, are defined as follows:

\begin{center}
$L_{\star}^{\mathbb{I}_j}(F) = \{s \in U : \mathbb{I}_{j}(s) \setminus F \in \mathcal{K}\}$.\\
$U^{\star\mathbb{I}_j}(F) = \{s \in U : \mathbb{I}_{j}(s) \cap F \not\in \mathcal{K}\}$.\\
$\bigtriangleup_{R}^{\star\mathbb{I}_j}(F) = U^{\star\mathbb{I}_j}(F) \setminus L_{\star}^{\mathbb{I}_j}(F)$.\\
$\mathcal{M}_{R}^{\star\mathbb{I}_j^\mathcal{K}}(F) = \frac{\mid L_{\star}^{\mathbb{I}_j}(F)\cap F\mid}{\mid U^{\star\mathbb{I}_j}(F)\cup F\mid}$, where $\mid U^{\star\mathbb{I}_j}(F)\cup F\mid\neq0$.\\
\color{black}
\end{center}
\end{definition}

\begin{definition} \label{d12} \cite{Al-shami 2024} Let $R$ be a relation on $U$, and let $\mathcal{K}$ be an ideal on $U$. The $\mathbb{I}^{\mathcal{K}}_{j}$-neighborhoods of an element $s \in U$ are specified as follows:
\begin{enumerate}
\item $\mathbb{I}^{\mathcal{K}}_{a}(s) = \{t \in U : \omega_{a}(t) \cap \omega_{a}(s) \not\in\mathcal{K}\}$.
\item  $\mathbb{I}^{\mathcal{K}}_{b}(s) = \{t \in U : \omega_{b}(t) \cap \omega_{b}(s) \not\in\mathcal{K}\}$.
\item  $\mathbb{I}^{\mathcal{K}}_{i}(s) = \mathbb{I}^{\mathcal{K}}_{a}(s) \cap \mathbb{I}^{\mathcal{K}}_{b}(s)$.
\item  $\mathbb{I}^{\mathcal{K}}_{u}(s) = \mathbb{I}^{\mathcal{K}}_{a}(s) \cup \mathbb{I}^{\mathcal{K}}_{b}(s)$.
\item  $\mathbb{I}^{\mathcal{K}}_{<a>}(s) = \{t \in U : \omega_{<a>}(t) \cap \omega_{<a>}(s) \not\in\mathcal{K}\}$.
\item  $\mathbb{I}^{\mathcal{K}}_{<b>}(s) = \{t \in U : \omega_{<b>}(t) \cap \omega_{<b>}(s) \not\in\mathcal{K}\}$.
\item  $\mathbb{I}^{\mathcal{K}}_{<i>}(s) = \mathbb{I}^{\mathcal{K}}_{<a>}(s) \cap \mathbb{I}^{\mathcal{K}}_{<b>}(s)$.
\item  $\mathbb{I}^{\mathcal{K}}_{<u>}(s) = \mathbb{I}^{\mathcal{K}}_{<a>}(s) \cup \mathbb{I}^{\mathcal{K}}_{<b>}(s)$.
\end{enumerate}
\end{definition}

Note that if $\mathcal{K} = \{\emptyset\}$, then Definition \ref{d12} is equivalent to Definition \ref{d3}. Consequently, the work presented in \cite{Al-shami 2024} can be regarded as a genuine extension of the results in  \cite{Al-shami 2021}. \newline

The notion of the "$\mathbb{I}$-Generalized approximation space" (abbreviated as $\mathbb{I}$-$G$ approximation space) was not defined in \cite{Al-shami 2024}, although it is referenced throughout the definitions and results. Therefore, we identified it as follows:


\begin{definition} \label{d111} Let $R$ be a relation on $U$ and $\mathcal{K}$ be an ideal on $U$. Then, the triple $(U, R, \mathcal{K})$ referred to as an $\mathbb{I}$-$G$ approximation space.
\end{definition}

\begin{theorem}\label{t104} \cite{Al-shami 2024} Let $(U, R, \mathcal{K})$ be an $\mathbb{I}$-$G$ approximation space with $s\in U$.
 If $R$ is a reflexive relation, then $\mathbb{I}^{\mathcal{K}}_{<j>}(s) \subseteq \mathbb{I}^{\mathcal{K}}_{j}(s)$, for each $j \in \{a, b, i, u\}$.
\end{theorem}

\begin{definition} \label{d13} \cite{Al-shami 2024} Let $R$  and $\mathcal{K}$ denote a binary relation and an ideal on a nonempty set $U$, respectively. The improved operators (lower and upper), boundary region, and accuracy of a nonempty subset $F$ of $U$ derived from $R$ and $\mathcal{K}$ are defined as follows:
\begin{center}
$R_{\star}^{\mathbb{I}_{j}^{\mathcal{K}}}(F) = \{s \in U : \mathbb{I}^{\mathcal{K}}_{j}(s) \setminus F \in \mathcal{K}\}$.\\
$R^{\star\mathbb{I}_{j}^{\mathcal{K}}}(F) = \{s \in U : \mathbb{I}^{\mathcal{K}}_{j}(s)\cap F \not\in \mathcal{K}\}$.\\
$BND_{R}^{\star\mathbb{I}_{j}^{\mathcal{K}}}(F) = R^{\star\mathbb{I}_{j}^{\mathcal{K}}}(F) \setminus R_{\star}^{\mathbb{I}_{j}^{\mathcal{K}}}(F)$.\\
$ACC_{R}^{\star\mathbb{I}_{j}^{\mathcal{K}}}(F) = \frac{\mid R_{\star}^{\mathbb{I}_{j}^{\mathcal{K}}}(F)\cap F\mid}{\mid R^{\star\mathbb{I}_{j}^{\mathcal{K}}}(F)\cup F\mid}$, where $\mid R^{\star\mathbb{I}_{j}^{\mathcal{K}}}(F)\cup F\mid\neq0$.\\
\end{center}
\end{definition}

\begin{remark} \label{emptycondition}
\begin{enumerate}
\item According to Pawlak's rough set theory principles, as described in \cite{Pawlak 1982}, if the lower approximation of the empty set $\emptyset$ is equal to its upper approximation, which is also $\emptyset$, then $\emptyset$ is a definable (exact) set, and thus, the accuracy measure of the empty set is 1. If this condition is not met, then $\emptyset$ is an undefinable (rough) set, and its accuracy measure is not equal to 1. Additionally, in the definition of the accuracy measure, there is a condition that the upper approximation must not be equal to $\emptyset$ to avoid division by zero. Therefore, if the upper approximation is equal to $\emptyset$, the accuracy measure in this case is considered an indefinite quantity.
\item Accordingly, in Definition \ref{d13}, $\emptyset$ is a definable (exact) set if $R_{\star}^{\mathbb{I}_{j}^{\mathcal{K}}}(\emptyset) = R^{\star\mathbb{I}_{j}^{\mathcal{K}}}(\emptyset) = \emptyset$, which implies that $ACC_{R}^{\star\mathbb{I}_{j}^{\mathcal{K}}}(\emptyset) = 1$. Otherwise, $\emptyset$ is an undefinable (rough) set, and therefore, $ACC_{R}^{\star\mathbb{I}_{j}^{\mathcal{K}}}(\emptyset)$ is considered an indefinite quantity.
\end{enumerate}
\end{remark}

Theorem \ref{t1} presents a significant result for generating a general topology from any given neighborhood, a method initially proposed by Abo Khadra et al. (2007) \cite{khadra2007} and later extended  in (2014) by Abd El-Monsef et al. \cite{Abd El-Monsef 2014}. This approach has subsequently been employed by other researchers to develop various topological structures, as illustrated in Theorem \ref{t1}.
\color{black}
\begin{theorem}\label{t1} Let $U$ be a universe and $j \in \Omega$. Then,

\begin{enumerate}
\item The fundamental method \cite{khadra2007,Abd El-Monsef 2014}: $$\top^{\omega_j} = \{F \subseteq U :  \omega_{j}(s) \subseteq F, \forall s \in F\}$$ is a topology on $U$.
\item Adhesion topologies \cite{Atef 2020, El-Bably 2021, Nawar 2021}: $\top^{\rho_j}= \{F \subseteq U :  \rho_{j}(s) \subseteq F, \forall s \in F\}$ is a topology on $U$.
\item $\mathbb{I}_{j}$-topologies \cite{Al-shami 2021}: $\top^{\mathbb{I}_j} = \{F \subseteq U :  \mathbb{I}_{j}(s) \subseteq F, \forall s \in F\}$ is a topology on $U$.
\end{enumerate}
\end{theorem}

\begin{definition} \label{d10} \cite{Al-shami 2021}
Let $\top^{\mathbb{I}_j}$ be a topology on $U$ as defined by the theorem above, for all $j \in \Omega$ and any $F\subseteq U$. The interior and closure operators of $F$ in $(U,  \top^{\mathbb{I}_j})$, denoted by  $\underline{\top^{\mathbb{I}_j}}(F)$ and $\overline{\top^{\mathbb{I}_j}}(F)$, are referred to as the $\top^{\mathbb{I}_j}$-lower approximation and $\top^{\mathbb{I}_j}$-upper approximation, respectively.
\end{definition}

\begin{definition} \label{d11} \cite{Al-shami 2021} The $\top^{\mathbb{I}_j}$-boundary and $\top^{\mathbb{I}_j}$-accuracy induced by a topological space $(U, \top^{\mathbb{I}_j})$ are  given by $BND^{\top^{\mathbb{I}_j}}(F) =  \overline{\top^{\mathbb{I}_j}}(F) \setminus  \underline{\top^{\mathbb{I}_j}}(F)$ and $ACC^{\top^{\mathbb{I}_j}}(F) = \frac{|\underline{\top^{\mathbb{I}_j}}(F)|}{|\overline{\top^{\mathbb{I}_j}}(F)|}$, where $|\overline{\top^{\mathbb{I}_j}}(F)|\neq0$.
\end{definition}
\color{black}

\begin{theorem}\label{t103} \cite{Al-shami 2024} Let $(U, R, \mathcal{K})$ be an $\mathbb{I}$-$G$ approximation space and $s\in U$. Then, for all $j \in \Omega$, the collection $\tau^{{\mathbb{I}_{j}^{\mathcal{K}}}} = \{F \subseteq U :  {\mathbb{I}_{j}^{\mathcal{K}}}(s) \setminus F \in \mathcal{K}, \forall s \in F\}$ forms a topology on $U$.
\end{theorem}

If $\mathcal{K} = \{\emptyset\}$ in Theorem 5.2 \cite{Al-shami 2024}, then the resulting topologies are equivalent to those presented in
Theorem 2.11 \cite{Al-shami 2021}. Thus, the  work in \cite{Al-shami 2024} serves as a meaningful extension of the work in  \cite{Al-shami 2021}.

\begin{theorem}\label{t1111} \cite{Al-shami 2021b} Let $(U, R, \mathcal{K})$ be an $\rho$-$G$ approximation space and $s\in U$. Then, for all $j \in \Omega$, the collection $\tau^{{\rho_{j}^{\mathcal{K}}}} = \{F \subseteq U :  \rho_{j}(s) \setminus F \in \mathcal{K}, \forall s \in F\}$ is a topology on $U$.
\end{theorem}

\section{Main results}

This section serves as the focal point of the research, where we critically analyze the errors identified in the study conducted by Al-Shami and Hosny \cite{Al-shami 2024}. We substantiate these inaccuracies by presenting multiple counterexamples and offer corrections for the identified errors. Moreover, we clarify and rectify mistakes in the proofs of certain results while addressing inconsistencies in the examples and comparison tables presented in their study. \newline

The following theorem, designated as Theorem 3.4 in \cite{Al-shami 2024}, contains errors that require correction.

\begin{theorem}\label{t3} \cite{Al-shami 2024} Let $(U, R, \mathcal{K})$ be an $\mathbb{I}$-$G$ approximation space with $s\in U$. The following statements hold:
\begin{enumerate}
\item  $\mathbb{I}^{\mathcal{K}}_{i}(s) \subseteq \mathbb{I}^{\mathcal{K}}_{a}(s) \cap \mathbb{I}^{\mathcal{K}}_{b}(s) \subseteq \mathbb{I}^{\mathcal{K}}_{a}(s) \cup \mathbb{I}^{\mathcal{K}}_{b}(s) \subseteq \mathbb{I}^{\mathcal{K}}_{u}(s)$.
\item  $\mathbb{I}^{\mathcal{K}}_{<i>}(s) \subseteq \mathbb{I}^{\mathcal{K}}_{<a>}(s) \cap \mathbb{I}^{\mathcal{K}}_{<b>}(s) \subseteq \mathbb{I}^{\mathcal{K}}_{<a>}(s) \cup \mathbb{I}^{\mathcal{K}}_{<b>}(s) \subseteq \mathbb{I}^{\mathcal{K}}_{<u>}(s)$.
\item  If $R$ is reflexive relation, then $\rho_{j}(s) \cup \omega_{j}(s) \subseteq \mathbb{I}^{\mathcal{K}}_{j}(s)$, for all $j \in \Omega$.
\item  If $R$ is serial, then $\rho_{j}(s) \subseteq \mathbb{I}^{\mathcal{K}}_{j}(s)$, for all $j \in \Omega$.
\item  If $R$ is transitive, then $\mathbb{I}^{\mathcal{K}}_{j}(s) \subseteq \mathbb{I}^{\mathcal{K}}_{<j>}(s)$, for each $j \in \{a, b, i, u\}$.
\end{enumerate}
\end{theorem}

\begin{remark}
Firstly, items 1 and 2 are incorrect, as the equalities cannot be replaced by subset symbols according to Definition \ref{d2}. The following examples demonstrate errors in the remaining items of the preceding theory.
\end{remark}
\color{black}
\begin{example}\label{e1} Let $R = \{(p, s), (p, t), (q, t), (t, q), (p, p), (q, q), (s, s), (t, t)\}$ be a reflexive relation on $U=\{p, q, s, t\}$. Therefore,
$\omega_{a}(t)=\{q, t\}$, $\omega_{b}(t)=\{p, q, t\}$, $\omega_{i}(t)=\{q, t\}$, $\omega_{u}(t)=\{p, q, t\}$,\\
 $\omega_{<a>}(t)=\{t\}$, $\omega_{<b>}(t)=\{q, t\}$, $\omega_{<i>}(t)=\{t\}$, $\omega_{<u>}(t)=\{q, t\}$,\\
 $\rho_{a}(t)=\{q, t\}$, $\rho_{b}(t)=\{t\}$, $\rho_{i}(t)=\{t\}$, $\rho_{u}(t)=\{q, t\}$,\\
 $\rho_{<a>}(t)=\{t\}$, $\rho_{<b>}(t)=\{q, t\}$, $\rho_{<i>}(t)=\{t\}$, $\rho_{<u>}(t)=\{q, t\}$.\\
 If $\mathcal{K} = \{\emptyset, \{t\}\}$, then
$\mathbb{I}^{\mathcal{K}}_{a}(t)=\{q, t\}$, $\mathbb{I}^{\mathcal{K}}_{b}(t)=U$, $\mathbb{I}^{\mathcal{K}}_{i}(t)=\{q, t\}$, $\mathbb{I}^{\mathcal{K}}_{u}(t)=U$,\\
 $\mathbb{I}^{\mathcal{K}}_{<a>}(t)=\emptyset$, $\mathbb{I}^{\mathcal{K}}_{<b>}(t)=\{q, t\}$, $\mathbb{I}^{\mathcal{K}}_{<i>}(t)=\emptyset$, $\mathbb{I}^{\mathcal{K}}_{<u>}(t)=\{q, t\}$.\\
 Clearly, $\rho_{j}(t) \nsubseteq \mathbb{I}^{\mathcal{K}}_{j}(t)$, and $\omega_{j}(t) \nsubseteq \mathbb{I}^{\mathcal{K}}_{j}(t)$, when $j\in\{<a>, <i>\}$.
\end{example}

\begin{example}\label{e18} Let $R = \{(p, p), (s, p), (t, p), (q, t), (t, q), (t, t)\}$ be a serial relation on $U=\{p, q, s, t\}$.
As a result,
  $\omega_{a}(p)=\{p\}$, $\omega_{a}(q)=\{t\}$, $\omega_{a}(s)=\{p\}$, $\omega_{a}(t)=\{p, q, t\}$,\\
  $\rho_{a}(p)=\{p, s\}$, $\rho_{a}(q)=\{q\}$, $\rho_{a}(s)=\{p, s\}$, $\rho_{a}(t)=\{t\}$,\\
 $\mathbb{I}_{a}(p)=\{p, s, t\}$, $\mathbb{I}_{a}(q)=\{q, t\}$, $\mathbb{I}_{a}(s)=\{p, s, t\}$, $\mathbb{I}_{a}(t)=U$.\\
 If $\mathcal{K} = \{\emptyset, \{t\}\}$, then $\mathbb{I}^{\mathcal{K}}_{a}(p)=\{p, s, t\}$, $\mathbb{I}^{\mathcal{K}}_{a}(q)=\emptyset$, $\mathbb{I}^{\mathcal{K}}_{a}(s)=\{p, s, t\}$, $\mathbb{I}^{\mathcal{K}}_{a}(t)=\{p, s, t\}$,\\
  Clearly, $\rho_{a}(q) \nsubseteq \mathbb{I}^{\mathcal{K}}_{a}(q)$.
\end{example}

\begin{example}\label{e4}Let $R = \{(p, s), (p, t), (p, q), (t, q), (t, t)\}$ be a transitive relation on $U=\{p, q, s, t\}$. Hence,
$\omega_{b}(p)=\emptyset$, $\omega_{b}(q)=\{p, t\}$, $\omega_{b}(s)=\{p\}$, $\omega_{b}(t)=\{p, t\}$,\\
 $\omega_{<b>}(p)=\{p\}$, $\omega_{<b>}(q)=\emptyset$, $\omega_{<b>}(s)=\emptyset$, $\omega_{<b>}(t)=\{p, t\}$, \\
 $\rho_{b}(p)=\{p\}$, $\rho_{b}(q)=\{q, t\}$, $\rho_{b}(s)=\{s\}$, $\rho_{b}(t)=\{q, t\}$,\\
 $\rho_{<b>}(p)=\{p\}$, $\rho_{<b>}(q)=\{q, s\}$, $\rho_{<b>}(s)=\{q, s\}$, $\rho_{<b>}(t)=\{t\}$.\\
 If $\mathcal{K}=\{\emptyset, \{t\}\}$, then
$\mathbb{I}^{\mathcal{K}}_{b}(p)=\emptyset$, $\mathbb{I}^{\mathcal{K}}_{b}(q)=\{q, s, t\}$, $\mathbb{I}^{\mathcal{K}}_{b}(s)=\{q, s, t\}$, $\mathbb{I}^{\mathcal{K}}_{b}(t)=\{q, s, t\}$,
$\mathbb{I}^{\mathcal{K}}_{<b>}(p)=\{p, t\}$, $\mathbb{I}^{\mathcal{K}}_{<b>}(q)=\emptyset$, $\mathbb{I}^{\mathcal{K}}_{<b>}(s)=\emptyset$, $\mathbb{I}^{\mathcal{K}}_{<b>}(t)=\{p, t\}$. Consequently, $\mathbb{I}^{\mathcal{K}}_{j}(x) \neq \mathbb{I}^{\mathcal{K}}_{<j>}(x)$, for all $x\in U$.\color{black}
\end{example}

\begin{example}\label{e5} Let $R = \{(p, p), (q, q), (s, s), (p, q), (p, s), (q, s)\}$ be a preorder relation on $U=\{p, q, s \}$. Hence,
$\omega_{a}(p)= U$, $\omega_{a}(q)=\{q, s\}$, $\omega_{a}(s)=\{s\}$, \\
 $\rho_{a}(p)= \{p\}$, $\rho_{a}(q)= \{q\}$, and $\rho_{a}(s)=\{s\}$.\\
If $\mathcal{K}=\{\emptyset, \{s\}\}$, then \color{black} $\mathbb{I}^{\mathcal{K}}_{a}(p)=\{p, q\}$, $\mathbb{I}^{\mathcal{K}}_{a}(q)=\{p, q\}$, $\mathbb{I}^{\mathcal{K}}_{a}(s)=\emptyset$. It is clear that $\omega_{a}(x) \nsubseteq \mathbb{I}^{\mathcal{K}}_{a}(x)$, for all $x\in U$. Additionally, $\rho_{a}(s) \nsubseteq \mathbb{I}^{\mathcal{K}}_{a}(s)$ \color{black}.
\end{example}

\begin{remark} \label{rem3-2}
\begin{enumerate}
\item For all $j \in \Omega$, $x\in U$, \color{black} $\rho_{j}(x) \nsubseteq \mathbb{I}^{\mathcal{K}}_{j}(x)$ if $R$ is serial (resp. reflexive, symmetric, transitive, preorder, or similarity), as shown in Examples \ref{e1}, \ref{e18}, \ref{e4}, and \ref{e5}.
\item For all $j \in \Omega$,  $x, y\in U$, \color{black} if $\omega_{j}(x) \cap \omega_{j}(y) \neq \emptyset$, this does not imply that $\omega_{j}(x) \cap \omega_{j}(y) \not\in \mathcal{K}$, as shown in Examples \ref{e1}, \ref{e18}, \ref{e4}, and \ref{e5}.
\item For all $j \in \Omega$, $x, y\in U$, \color{black} if $\mathbb{I}^{\mathcal{K}}_{j}(x) \cap \mathbb{I}^{\mathcal{K}}{j}(y) \neq \emptyset$, this does not imply that $\mathbb{I}^{\mathcal{K}}_{j}(x) \cap \mathbb{I}^{\mathcal{K}}_{j}(y) \not\in \mathcal{K}$, as shown in Examples \ref{e1}, \ref{e18}, \ref{e4}, and \ref{e5}.
\end{enumerate}
\end{remark}


Consequently, the corrected formulation of Theorem 3.4 from \cite{Al-shami 2024} is presented as follows:

\begin{theorem}\label{t3} Let $(U, R, \mathcal{K})$ be an $\mathbb{I}$-$G$ approximation space and $s, t\in U$. The following statements are valid:
\begin{enumerate}
\item  $ t\in \mathbb{I}^{\mathcal{K}}_{j}(s) \Leftrightarrow s \in \mathbb{I}^{\mathcal{K}}_{j}(t)$, for all $j$ $ \in \Omega$.
\item  If $R$ is reflexive, then $\mathbb{I}^{\mathcal{K}}_{<j>}(s)\subseteq \mathbb{I}^{\mathcal{K}}_{j}(s)$, for each $j \in \{a, b, i, u\}$.
\item  If $R$ is preorder, then $\mathbb{I}^{\mathcal{K}}_{<j>}(s)=\mathbb{I}^{\mathcal{K}}_{j}(s)$, for each $j \in \{a, b, i, u\}$.
\item  If $R$ is reflexive, then $\rho_{j}(s) \cap \omega_{j}(s) = \rho_{j}(s)$ and $\rho_{j}(s) \cup \omega_{j}(s) = \omega_{j}(s)$, for all $j \in \Omega$.
\item  If $R$ is symmetric, then $\mathbb{I}^{\mathcal{K}}_{i}(s) = \mathbb{I}^{\mathcal{K}}_{a}(s) = \mathbb{I}^{\mathcal{K}}_{b}(s) = \mathbb{I}^{\mathcal{K}}_{u}(s)$ and $\mathbb{I}^{\mathcal{K}}_{<i>}(s) = \mathbb{I}^{\mathcal{K}}_{<a>}(s) = \mathbb{I}^{\mathcal{K}}_{<b>}(s) = \mathbb{I}^{\mathcal{K}}_{<u>}(s)$.
\end{enumerate}
\end{theorem}
\begin{proof}
First, the proofs of items (1), (2), and (5) are provided in \cite{Al-shami 2024}. We will now prove the remaining items that are not stated in \cite{Al-shami 2024} as follows:

\noindent (3) According to \cite{Abu Gdairi 2024},  since $R$ is a preorder, we have $\omega_{<j>}(s) = \omega_{j}(s)$ for all $j \in \Omega$. Therefore, based on items (1) and (2), we conclude that $\mathbb{I}^{\mathcal{K}}_{<j>}(s) = \mathbb{I}^{\mathcal{K}}_{j}(s)$.

\noindent (4) As noted in  \cite{El-Bably 2021}, since $R$ is reflexive, we find that $\rho_{j}(s) \subseteq \omega_{j}(s)$ for all $j \in \Omega$. Thus, it follows that, $\rho_{j}(s) \cap \omega_{j}(s) = \rho_{j}(s)$ and $\rho_{j}(s) \cup \omega_{j}(s) = \omega_{j}(s)$ for all $j \in \Omega$.
\end{proof}

As noted in Remark \ref{rem3-2}, the proof of item (8) in Theorem 3.4 from \cite{Al-shami 2024} contains errors. To address this, we present the following lemma, along with a detailed proof (absent in \cite{Al-shami 2024}), as a correction to item (8) of Theorem 3.4 in \cite{Al-shami 2024} and also to support the proof of the Theorem \ref{t4}:

\begin{lemma} \label{l1} Let $R$ be a symmetric and transitive relation on $U$. If $s \in \omega_{j}(t)$, then $\omega_{j}(s) = \omega_{j}(t)$ for all $s, t \in U$ and $j \in \Omega$.
\end{lemma}
\begin{proof}
We provide the proof for the case $j = a$, noting that the other cases follow similarly. \\

First, let us assume that $s \in \omega_{a}(t)$. Then,
\begin{equation}  \label{Eq01}
tRs
\end{equation}
Since $R$ is symmetric, we have
\begin{equation}  \label{Eq02}
sRt
\end{equation}

Now, let $x \in \omega_{a}(s)$. Then,
\begin{equation}  \label{Eq03}
sRx
\end{equation}

From Equations \ref{Eq01} and \ref{Eq03}, and given that $R$ is transitive, it follows that $tRx$, which implies $x \in \omega_{a}(t)$. Hence, $\omega_{a}(s) \subseteq \omega_{a}(t)$.

Conversely, let $y \in \omega_{a}(t)$. Then,
\begin{equation}  \label{Eq04}
tRy
\end{equation}

by referring to Equations \ref{Eq02} and \ref{Eq04}, it follows that $sRy$ due to the transitivity of $R$.
This indicates that $y \in \omega_{a}(s)$, leading to the conclusion that $\omega_{a}(t) \subseteq \omega_{a}(s)$.

Thus, we arrive at the result $\omega_{a}(s) = \omega_{a}(t)$.
\end{proof}

The proof for item (8) of Theorem 3.4 in \cite{Al-shami 2024} is incorrect. Accordingly, we will rectify  and rephrase it as follows:

\begin{theorem}\label{t4} \cite{Al-shami 2024} Let $R$ be a symmetric and transitive relation on $U$, and let $\mathcal{K}$ be an ideal on $U$. If $s \in U$ and  $j \in \Omega$, then the following holds:
\begin{enumerate}
\item $\mathbb{I}^{\mathcal{K}}_{j}(s) \subseteq \omega_{j}(s)$.

\item If $s \in \mathbb{I}^{\mathcal{K}}_{j}(t)$, then $\mathbb{I}^{\mathcal{K}}_{j}(s) \subseteq \mathbb{I}^{\mathcal{K}}_{j}(t)$.
\end{enumerate}
\end{theorem}
\begin{proof}
\begin{enumerate}
\item Suppose $j = a$. Let $p \in \mathbb{I}^{\mathcal{K}}_{j}(s)$. Then, we have $\omega_{a}(p) \cap \omega_{a}(s) \not\in \mathcal{K}$, which implies that $\omega_{a}(p) \cap \omega_{a}(s) \neq \emptyset$. Therefore, there exists an element $q \in \omega_{a}(p) \cap \omega_{a}(s)$, such that, $pRq$ and $sRq$. Given that  $R$ is symmetric and transitive, it follows that $sRp$. Consequently, we can conclude that $p \in \omega_{a}(s)$. Thus, we have established that $\mathbb{I}^{\mathcal{K}}_{j}(s) \subseteq \omega_{j}(s)$.

\item Now, suppose $s \in \mathbb{I}^{\mathcal{K}}_{j}(t)$. Then, we have $\omega_{j}(s) \cap \omega_{j}(t)$$\not\in$$\mathcal{K}$, which implies that $\omega_{j}(s) \cap \omega_{j}(t)$$\neq$$\emptyset$. Consequently, there exists an element $q$ such that $q \in \omega_{j}(s)$ and $q \in \omega_{j}(t)$. Since $R$ is symmetric and transitive, and by Lemma \ref{l1},  it follows that $\omega_{j}(q) = \omega_{j}(s) = \omega_{j}(t)$. Thus,
    \begin{equation}  \label{Eq05}
    \omega_{j}(s) \cap \omega_{j}(t) = \omega_{j}(s) = \omega_{j}(t) \not\in \mathcal{K}
    \end{equation}

Next, let $p \in \mathbb{I}^{\mathcal{K}}_{j}(s)$. Then, $\omega_{j}(p) \cap \omega_{j}(s) \not\in \mathcal{K}$, which implies that $\omega_{j}(p) \cap \omega_{j}(s) \neq \emptyset$. Hence, there exists an element $m$ such that $m \in \omega_{j}(s)$ and $m \in \omega_{j}(t)$. Since $R$ is symmetric and transitive, and by Lemma \ref{l1}, it follows that $\omega_{j}(p) = \omega_{j}(s) = \omega_{j}(m)$. Thus,
    \begin{equation}  \label{Eq06}
    \omega_{j}(p) \cap \omega_{j}(s) = \omega_{j}(s)
    \end{equation}
    Therefore, from Equations \ref{Eq05} and \ref{Eq06}, we obtain $\omega_{j}(p) \cap \omega_{j}(t) \not\in \mathcal{K}$, which implies $p \in \mathbb{I}^{\mathcal{K}}_{j}(t)$.
\end{enumerate}
\end{proof}

The converse of Theorem \ref{t4} (1) is not generally valid, as demonstrated by the following example.

\begin{example}\label{e2} Let $R = \{(t, t)\}$ be a symmetric and transitive relation on $U=\{p, q, s, t\}$. If $\mathcal{K} = \{\emptyset, \{t\}\}$, then we find that $\omega_{a}(t)=\{t\}$, $\mathbb{I}^{\mathcal{K}}_{a}(t)=\emptyset$. Hence, $\omega_{j}(t) \nsubseteq \mathbb{I}^{\mathcal{K}}_{j}(t)$.
\end{example}

Table 1 in \cite{Al-shami 2024} introduces kinds of neighborhoods such as $\omega_{j}$-neighborhoods,  $\rho_{j}$-neighborhoods, $\mathbb{I}_{j}$-neighborhoods, and $\mathbb{I}^{\mathcal{K}}_{j}$-neighborhoods, for each $j \in \Omega$. Here, we correct several errors in Table 1, that create ambiguity in the intended meaning. Since it is possible to equal the neighborhood of a point in the universe to the $\empty set$, but it is not equal to  $\{\emptyset\}$.\\


Table 2 in \cite{Al-shami 2024} provides a comparative analysis of the lower and upper operators, along with accuracy measure results for any subset of $U$, with $\{a, b, i, u\}$. According to Remark \ref{emptycondition}, several observations and corrections have been noted based on this table.

\begin{enumerate}
\item Since $R_{\star}^{\mathbb{I}_{a}^{\mathcal{K}}}(\emptyset)=\{q, s, t\}$, it follows that $R^{\star\mathbb{I}_{a}^{\mathcal{K}}}(\emptyset)=\emptyset$. Consequently, $ACC_{R}^{\star\mathbb{I}_{a}^{\mathcal{K}}}(\emptyset)$ is not zero; rather, it is an indefinite quantity. By the same manner, $ACC_{R}^{\star\mathbb{I}_{b}^{\mathcal{K}}}(\emptyset)$, $ACC_{R}^{\star\mathbb{I}_{i}^{\mathcal{K}}}(\emptyset)$ and $ACC_{R}^{\star\mathbb{I}_{u}^{\mathcal{K}}}(\{p\})$ should be indefinite quantities.
\item Since $R_{\star}^{\mathbb{I}_{u}^{\mathcal{K}}}(\emptyset)= R^{\star\mathbb{I}_{u}^{\mathcal{K}}}(\emptyset)=\emptyset $. Therefore, $ACC_{R}^{\star\mathbb{I}_{u}^{\mathcal{K}}}(\emptyset)$ should be 1, not zero.
\item $R_{\star}^{\mathbb{I}_{b}^{\mathcal{K}}}(\emptyset)=\{q\}$, this set is not equal to $\emptyset$. Thus, $ACC_{R}^{\star\mathbb{I}_{b}^{\mathcal{K}}}(\emptyset)$ is also an indefinite quantity.
\item $ACC_{R}^{\star\mathbb{I}_{i}^{\mathcal{K}}}(\emptyset) = 1$.
\item $R_{\star}^{\mathbb{I}_{u}^{\mathcal{K}}}(\{p\})=\{q\}$, this set is not equal to $\emptyset$.
 \color{black}
\end{enumerate}
\color{black}
Item (3) of Proposition 4.3 in \cite{Al-shami 2024}, which asserts that $R^{\star\mathbb{I}_{a}^{\mathcal{K}}}(U)\supseteq U$, is incorrect. As illustrated in Table 2 of \cite{Al-shami 2024}, we find that $R^{\star\mathbb{I}_{a}^{\mathcal{K}}}(U)=\{p\}\nsupseteq U$. Several observations and corrections have been recorded based on this table.\\

Remark 4.4 in \cite{Al-shami 2024} further claims that the inclusion relations in Proposition 4.3 (3) are generally proper. However, this assertion is also inaccurate, as  $R^{\star\mathbb{I}_{a}^{\mathcal{K}}}(U)\subseteq U$ does not hold in this context.\color{black}\\

According to Remark \ref{emptycondition}, Table 3 in \cite{Al-shami 2024} provides a comparison between results obtained from the proposed Definition 4.1 in \cite{Al-shami 2024} and those from Definition 2.7 in \cite{Al-shami 2021} for the case $j = a$. This comparison reveals several observations and necessary corrections, as outlined below:
\begin{enumerate}
\item Given that $R_{\star}^{\mathbb{I}_{a}^{\mathcal{K}}}(\emptyset)=\{q, s, t\}$, $R^{\star\mathbb{I}_{a}^{\mathcal{K}}}(\emptyset)=\emptyset$, it follows that $ACC_{R}^{\star\mathbb{I}_{a}^{\mathcal{K}}}(\emptyset)$ is not zero but an indefinite quantity.
\item $ACC_{R}^{\star\mathbb{I}_{a}}(\emptyset)$ should be 1, not zero.\\
\end{enumerate}

According to Remark \ref{emptycondition}, Table 4 in \cite{Al-shami 2024} offers a comparative analysis of results derived from the proposed Definition 4.1 from \cite{Al-shami 2024}, alongside those of Definition 2 from \cite{Yao 1998}, Definition 3 from \cite{Allam2005}  and Definitions 4.3, 4.5 in \cite{Atef 2020} for the case $j = b$. Several omissions and errors were identified in this comparison, as outlined below.
\begin{enumerate}
\item $R_{\star}^{\mathbb{I}_{b}^{\mathcal{K}}}(\emptyset)=\{q\}$, which is not equal to $\emptyset$. Therefore, $\emptyset$ is rough which implies $ACC_{R}^{\star\mathbb{I}_{b}^{\mathcal{K}}}(\emptyset)$ is undefined.
\item $R_{\star}^{w_{b}}(\emptyset)=\emptyset$, not $\{q\}$. Thus, $R_{\star}^{w_{b}}(\emptyset)= R^{\star\omega_b}(\emptyset)=\emptyset$ which means that $\emptyset$ is exact and hence $ACC_{R}^{\star w_{b}}(\emptyset) = 1$.
\item $R_{\star}^{w_{b}}(\{p\})=\{p, s\}$, differing from $\{p, q, s\}$.
\item $R_{\star}^{w_{b}}(\{q\})=\emptyset$, not $\{q\}$. Additionally, $ACC_{R}^{\star w_{b}}(\{q\})= 0$.
\item $R_{\star}^{w_{b}}(\{s\})=\emptyset$, differing from $\{q\}$.
\item $R_{\star}^{w_{b}}(\{p, q\})=\{p, s\}$, not equal to $U$. Also, $ACC_{R}^{\star\omega_{b}}(\{p, q\})=\frac{1}{4}$.
\item $R_{\star}^{w_{b}}(\{p, s\})=\{p, s\}$, differing from $\{p, q, s\}$.
\item $R_{\star}^{w_{b}}(\{q, s\})=\emptyset$, it does not equal to $\{q\}$.
Also, $ACC_{R}^{\star w_{b}}(\{q, s\})= 0$.
 \item $R_{\star}^{w_{b}}(\{p, q, s\})=\{p, s\}$, differing from $U$. Also, $ACC_{R}^{\star w_{b}}(\{p, q, s\})= \frac{1}{2}$.
 \item $R_{\star}^{\rho_{b}}(\emptyset)=\emptyset$, differing from $\{t\}$. Therefore, $R_{\star}^{\rho_{b}}(\emptyset) = R^{\star \rho_{b}}(\emptyset) = \emptyset$ which implies $\emptyset$ represents an exact set and hence $ACC_{R}^{\star \rho_{b}}(\emptyset) = 1$.
 \item $R^{\star \rho_{b}}(U)=U$, not $\{p, q, s\}$.
 \item  $R_{\star}^{\rho_{b}}(\{p\})=\emptyset$, not $\{t\}$, and $R^{\star \rho_{b}}(\{p\})=\{p, s\}$, differing from $\{p, q\}$.
 \item  $R_{\star}^{\rho_{b}}(\{q\})=\{q\}$, not $\{q, t\}$.
  \item  $R_{\star}^{\rho_{b}}(\{s\})=\emptyset$, it does not equal to $\{t\}$.
  \item $R^{\star \rho_{b}}(\{t\})=\{t\}$, not $\emptyset$.
 \item  $R_{\star}^{\rho_{b}}(\{p, q\})=\{q\}$, not $\{q, t\}$.
\item  $R_{\star}^{\rho_{b}}(\{p, s\})=\{p, s\}$, differing from $\{p, s, t\}$.
\item  $R_{\star}^{\rho_{b}}(\{p, t\})=\{t\}$, not $\{p, s, t\}$, and $R^{\star \rho_{b}}(\{p, t\})=\{p, s, t\}$, differing from $\{p, s\}$. In addition, $ACC_{R}^{\star \rho_{b}}(\{p, t\})= \frac{1}{3}$.
 \item  $R_{\star}^{\rho_{b}}(\{q, s\})=\{q\}$, it does not equal to $\{q, t\}$.
 \item $R^{\star \rho_{b}}(\{q, t\})=\{q, t\}$, it does not equal to $\{q\}$.
  \item $R^{\star \rho_{b}}(\{s, t\})=\{p, s, t\}$, not $\{p, s\}$.
 \item  $R_{\star}^{\rho_{b}}(\{p, q, s\})=\{p, q, s\}$, differing from $U$.
\item $R^{\star \rho_{b}}(\{p, q, t\})=U$, not $\{p, q\}$. Also, $ACC_{R}^{\star \rho_{b}}(\{p, q, t\})= \frac{1}{2}$.
\item  $R_{\star}^{\rho_{b}}(\{p, s, t\})=\{p, s, t\}$, not $\{p, s\}$, and $R^{\star \rho_{b}}(\{p, s, t\})=\{p, s, t\}$, differing from $\{p, s\}$. In addition, $ACC_{R}^{\star \rho_{b}}(\{p, s, t\})= 1$.
 \item $R^{\star \rho_{b}}(\{q, s, t\})=U$, differing from  $\{p, q\}$.
 \color{black}
\end{enumerate}

\begin{remark} In the proof of Theorem 5.1, the symbol $\mathcal{I}$ should be modified into the symbol $\mathcal{K}$.
\end{remark}

The items (5) and (8) of Proposition 5.4 in \cite{Al-shami 2024} do not hold in general, as demonstrated by the following example.

\begin{example}\label{e3} Let $R = \{(p, p), (p, s), (p, t), (q, t), (q, q), (s, s), (t, q), (t, t)\}$ be a reflexive relation on $U=\{p, q, s, t\}$.

    If $\mathcal{K} = \{\emptyset, \{s\}, \{t\}, \{s, t\}\}$, then $\tau^{{\mathbb{I}_{<a>}^{\mathcal{K}}}} =  2^{U}$ ($2^{U}$ is a power set of $U$) and $\tau^{{\mathbb{I}_{a}^{\mathcal{K}}}} = \{\{p\}, \{q\}, \{s\}, \{p, q\}, \{p, \\
     s\}, \{q, s\}, \{q, t\}, \{p, q, s\}, \{p, q, t\}, \{q, s, t\}, \emptyset, U\}$.
\color{black}
So, $\tau^{{\mathbb{I}_{<a>}^{\mathcal{K}}}}\nsubseteq \tau^{{\mathbb{I}_{a}^{\mathcal{K}}}}$.
\end{example}

The proposition below provides a correction to item (5) of Proposition 5.4 in \cite{Al-shami 2024}:

\begin{proposition} \label{P3.1} Let $(U, R, \mathcal{K})$ be an $\mathbb{I}$-$G$ approximation space and $s\in U$. If $R$ is reflexive, then $\tau^{{\mathbb{I}_{j}^{\mathcal{K}}}}\subseteq\tau^{{\mathbb{I}_{<j>}^{\mathcal{K}}}}$, for each $j \in \{a, b, i, u\}$.

\end{proposition}
\begin{proof}
Let $F \in \tau^{{\mathbb{I}_{j}^{\mathcal{K}}}}$ and $s\in F$, then ${\mathbb{I}_{j}^{\mathcal{K}}}(s) \setminus F \in \mathcal{K}$. Since $R$ is a reflexive relation on $U$, then from Theorem \ref{t104}, $\mathbb{I}^{\mathcal{K}}_{<j>}(s) \subseteq \mathbb{I}^{\mathcal{K}}_{j}(s)$. Therefore, $\mathbb{I}^{\mathcal{K}}_{<j>}(s) \setminus F \in \mathcal{K}$ i.e $F \in \tau^{{\mathbb{I}_{<j>}^{\mathcal{K}}}}$.
\end{proof}

\begin{remark} If $R$ is reflexive, then  $\rho_{j}(x) \nsubseteq \mathbb{I}^{\mathcal{K}}_{j}(x)$, for all $j \in \Omega$, and for every $x\in U$. \color{black} Consequently, $\tau^{{\rho_{j}^{\mathcal{K}}}} \nsubseteq \tau^{{\mathbb{I}_{j}^{\mathcal{K}}}}$, $\forall j \in\Omega$.\\
\end{remark}

According to Remark \ref{emptycondition}, Table 5 in \cite{Al-shami 2024} presents a comparison of the boundary region and accuracy measure results for a set $F$ based on Definition 5.7, where $j \in \{a, b, i, u\}$. The quantities $ACC_{R}^{\tau^{\mathbb{I}^{\mathcal{K}}_{a}}}(\emptyset)$ (and similarly $ACC_{R}^{\tau^{\mathbb{I}^{\mathcal{K}}_{b}}}(\emptyset), ACC_{R}^{\tau^{\mathbb{I}^{\mathcal{K}}_{i}}}(\emptyset), ACC_{R}^{\tau^{\mathbb{I}^{\mathcal{K}}_{u}}}(\emptyset)$) should be 1, not zero. Besides, $j=r, j=l$ have been revised  to  $j=a, j=b$ respectively.
\color{black} \\

Table 6 in \cite{Al-shami 2024} compares the boundary region and accuracy measure results for a set $F$ based on the proposed Definition 5.7 and Definition 2.13 from \cite{Al-shami 2021} for $j = a$. The symbol $\underline{\tau^{\mathbb{I}_{b}}}(L)$ (and similarly $\overline{\tau^{\mathbb{I}_{b}}}(L), BND^{\tau^{\mathbb{I}_{b}}}(L), ACC^{\tau^{\mathbb{I}_{b}}}(L)$) should be corrected to $\underline{\tau^{\mathbb{I}_{a}}}(L)$ (and similarly $\overline{\tau^{\mathbb{I}_{a}}}(L), BND^{\tau^{\mathbb{I}_{a}}}(L), \\
ACC^{\tau^{\mathbb{I}_{a}}}(L)$). In addition, the quantities $ACC_{R}^{\tau^{\mathbb{I}^{\mathcal{K}}_{a}}}(\emptyset)$, and $ACC_{R}^{\tau^{\mathbb{I}_{a}}}(\emptyset)$ are indefinite. \color{black}\\

\section{Conclusion and Discussion}

This paper offers a thorough examination of the work by Al-shami and Hosny \cite{Al-shami 2024}, identifying and rectifying key errors in their definitions, results, and examples concerning $\mathbb{I}^{\mathcal{K}}_{j}$-neighborhoods and their application within rough set theory. Through detailed analysis, counterexamples, and revised proofs, we have identified inaccuracies in previous research and proposed essential corrections.
These corrections not only contribute to theoretical advancements but also reinforce the broader applicability of neighborhood systems in rough set theory. One of the significant contributions of this work is the identification and correction of erroneous examples and results, which were central to the claims made by Al-shami and Hosny. By providing corrected results and examples, we have enhanced the reliability and validity of neighborhood-based approximations within the framework of rough set theory. Furthermore, the refined concepts and methodologies introduced in this paper offer new avenues for future research, equipping scholars with more robust tools to address uncertainty and perform data analysis across various domains, including medical applications, as demonstrated by the improved examples. \\

\textbf{The key objectives achieved in this paper are as follows:}
\begin{enumerate}
\item Several key errors have been identified in Theorem 3.4, Proposition 4.3, and Proposition 5.4 from \cite{Al-shami 2024}. These inaccuracies have been demonstrated through various examples, including Examples \ref{e1}, \ref{e18}, \ref{e4}, \ref{e5}, \ref{e2} and \ref{e3}. Furthermore, corrections and clarifications have been provided for a number of incorrect results, such as the revised proof of Theorem 3.4 from
    \cite{Al-shami 2024}, which is accurately presented in Theorems \ref{t3} and \ref{t4} of this paper.
\item Corrected versions of the erroneous results in \cite{Al-shami 2024} have been provided, including Theorem \ref{t3} and Proposition \ref{P3.1}.
\item Errors within the concepts, definitions, examples, and tables presented in \cite{Al-shami 2024} were identified and corrected. This includes revisions to Definitions \ref{d1}, \ref{d2}, \ref{d111}, and \ref{d13}, as well as significant modifications to five tables and the resolution of 25 distinct errors across several examples.
\item New results and properties have been proposed and demonstrated, including Remark \ref{rem3-2} and Lemma \ref{l1}.
\end{enumerate}

These findings emphasize the importance of critically evaluating newly proposed methodologies and results in dynamic fields such as rough set theory. The corrections provided in this work ensure that the theoretical underpinnings of neighborhood systems remain strong and reliable for future research and practical applications. Moreover, the enhanced understanding of neighborhood-based rough approximations broadens their utility in decision-making, data analysis, and knowledge discovery across various disciplines.\\
${}$\\
\textbf{Acknowledgments} \\
The authors wish to express their sincere appreciation to the Editor-in-Chief, the Section Editor, and the reviewers for their thoughtful feedback and constructive suggestions, which were instrumental in refining and enhancing the quality of this work. \\
${}$\\
\textbf{Use of AI tools declaration} \\
The authors declare they have not used Artificial Intelligence (AI) tools in the creation of this article. \\
${}$\\
\textbf{Availability of data and material} Not applicable.\newline
${}$\\
\textbf{\ Competing interests} The authors declare that they have no competing interests.\newline
${}$\\
\textbf{Funding} There is no funding available.\\
${}$\\

 \end{document}